\documentclass[fontsize=12pt,a4paper,DIV=11]{scrartcl}
\usepackage[utf8]{inputenc}
\usepackage[T1]{fontenc}
\usepackage[fleqn]{amsmath}
\usepackage{amsthm,amstext,amssymb,amsfonts,mathptmx}
\usepackage[babel,english=british]{csquotes}
\usepackage[english]{babel}
\usepackage[numbers]{natbib}
\usepackage{caption}
\usepackage{booktabs,subcaption,calc,scrpage2,graphicx}
\usepackage{listings,listingsutf8,xcolor,enumerate,tikz}
\usepackage{hyperref}
\setkomafont{disposition}{\normalfont\bfseries}
\newcommand{\T}{\ensuremath{\mathrm{T}}}
\newcommand{\E}{\ensuremath{\mathrm{e}}}
\newcommand{\I}{\ensuremath{\mathrm{i}}}
\newcommand{\D}{\ensuremath{\,\mathrm{d}}}
\newcommand{\mat}[1]{\ensuremath{\mathbf{#1}}}
\renewcommand{\vec}[1]{\ensuremath{\mathbf{#1}}}

\addto\captionsenglish{}
\newtheorem{theorem}{Theorem}
\newtheorem{lemma}{Lemma}
\newtheorem{definition}{Definition}
\newtheorem{remark}{Remark}

\title{Multivariate Anisotropic Interpolation\\on the Torus}
\author{Ronny Bergmann\thanks{Corresponding author}\({\ \!}^{\ ,}\)\thanks{Department of Mathematics, University of Kaiserslautern, Paul-Ehrlich-Stra\ss e 31, D-67663 Kaiserslautern, Germany,  \href{mailto:bergmann@mathematik.uni-kl.de}{bergmann@mathematik.uni-kl.de}}%
\and %
J{\"u}rgen Prestin%
\thanks{Institute of Mathematics, University of L{\"u}beck, Ratzeburger Allee 160, D-23562 L{\"u}beck, Germany, \newline\href{mailto:prestin@math.uni-luebeck.de}{prestin@math.uni-luebeck.de}}
}
\date{January 17, 2014}

\hypersetup{pdfauthor=Ronny Bergmann and Jürgen Prestin,
          pdftitle={Multivariate anisotropic interpolation on the torus},
          pdfsubject={We investigate the error of periodic interpolation, when sampling a function on an arbitrary pattern on the torus. We generalize the periodic Strang-Fix conditions to an anisotropic setting and provide an upper bound for the error of interpolation. These conditions and the investigation of the error especially take different levels of smoothness along certain directions into account.},
          pdfcreator = {pdflatex and TextMate},
          plainpages=false,
          pdfstartview=FitH,       
          pdfview=FitH,            
          pdfpagemode=UseOutlines, 
          bookmarksnumbered=true, 
          bookmarksopen=false,     
          bookmarksopenlevel=0,   
          colorlinks=true,       
          linkcolor=blue,
          citecolor=blue!75!black,
          urlcolor=black}
\begin{document}
\maketitle
\begin{abstract}
\noindent\textbf{Abstract.} We investigate the error of periodic interpolation, when sampling a function on an arbitrary pattern on the torus. We generalize the periodic Strang-Fix conditions to an anisotropic setting and provide an upper bound for the error of interpolation.  These conditions and the investigation of the error especially take different levels of smoothness along certain directions into account.	
\end{abstract}

\section{Introduction} 
	\label{sec:Introduction}
	Approximation by equidistant translates of a periodic function was first investigated in the univariate case \cite{Delvos:1987,Locher:1981}. The multivariate case was developed in \cite{Sprengel:1997,Sprengel:1998,Sprengel:2000}, where  the introduction of the periodic Strang-Fix conditions enabled a unified way to the error estimates~\cite{Poe95,Poplau:1997}.

	Recently, many approaches such as contourlets~\cite{DoVetterli2005}, curvelets~\cite{Fadili2007} or shearlets~\cite{GuoLabate:2010}, analyze and decompose multivariate data by focusing on certain anisotropic features. A more general approach are wavelets with composite dilations~\cite{Guo:2006,Krishtal:2011}, which inherit an MRA structure similar to the classical wavelets. For periodic functions the multivariate periodic wavelet analysis~\cite{Bergmann:2013,GohLeeTeo:1999,LangemannPrestin:2010,Skopina1997,Skopina2000} is a periodic approach to such an anisotropic decomposition. The pattern \( \mathcal P(\mat{M}) \) as a basic ingredient to these scaling and wavelet functions models equidistant points with preference of direction, i.e., fixing one direction \( \vec{v}\in\mathbb R^d, \|\vec{v}\|=1 \), we obtain equidistant points along this direction in the pattern \( \mathcal P(\mat{M}) \), though other directions might have other point distances.

	This paper presents the interpolation on such patterns \( \mathcal P(\mat{M}) \), where \( \mat{M}\in\mathbb Z^{d\times d} \), \( d\in\mathbb N \), is a regular integer matrix. In order to derive an upper bound for the interpolation error, we introduce function spaces \( A_{\mat{M},q}^{\alpha} \), where each function is of different directional smoothness due to decay properties of the Fourier coefficients imposed. The periodic Strang-Fix conditions can be generalized to this anisotropic setting, characterizing and quantifying the reproduction capabilities of a fundamental interpolant with respect to a certain set of trigonometric polynomials. Such a fundamental interpolant can then be used for approximation, where the error can be bounded for the functions having certain directional smoothness, i.e., the space \( A_{\mat{M},q}^{\alpha} \).

	The rest of the paper is organized as follows: In Section~\ref{sec:pre} we introduce the basic preliminary notations of the pattern \( \mathcal P(\mat{M}) \), the corresponding discrete Fourier transform \( \mathcal F(\mat{M}) \), and the spaces \( A_{\mat{M},q}^\alpha \). Section \ref{sec:IPSFC} is devoted to the interpolation problem on the pattern \( \mathcal P(\mat{M}) \) and the ellipsoidal periodic Strang-Fix conditions, which generalize the periodic Strang-Fix conditions to an anisotropic setting. For this interpolation, we derive an upper bound for interpolation with respect to \( A_{\mat{M},q}^\alpha \) in Section~\ref{sec:IPErr}. Finally, in Section~\ref{sec:Example} we provide an example that the ellipsoidal Strang-Fix conditions are fulfilled by certain periodized 3-directional box splines and their higher dimensional analogues.
%
%
%
\section{Preliminaries}\label{sec:pre} 
	%
	%
	\subsection{Patterns} 
		Let \( d\in\mathbb N \). For a regular integral matrix \( \mat{M}\in\mathbb Z^{d\times d} \) and two vectors \( \vec{h},\vec{k}\in\mathbb Z^d \) we write \( \vec{h}\equiv\vec{k}\bmod\mat{M} \) if there exists a vector \( \vec{z}\in\mathbb Z^d \) such that \( \vec{h} = \vec{k}+\mat{M}\vec{z} \). The set of congruence classes
		\[
		[\vec{h}]_\mat{M} := \bigl\{\vec{k}\in\mathbb Z^d\,;\, \vec{k}\equiv \vec{h} \bmod\mat{M}\bigr\},\quad\vec{h}\in\mathbb Z^d,
		\] forms a partition of \( \mathbb Z^d \) and using the addition \( [\vec{h}]_\mat{M}+[\vec{k}]_\mat{M} := [\vec{h}+\vec{k}]_\mat{M} \), we obtain the generating group \( \bigl(\mathcal G(\mat{M}), +\bigr) \), where the generating set \( \mathcal G(\mat{M}) \) is any set of congruence class representatives. If we apply the congruence with respect to the unit matrix \[
		{\mathbf E}_d := \bigl(\delta_{i,j}\bigr)_{i,j=1}^d
		\in\mathbb R^{d\times d}
		,
		\quad\text{where}\
		\delta_{i,j} := \begin{cases}
			1&\mbox { if } i=j,\\
			0&\mbox{ else, }
		\end{cases}
		\] denotes the Kronecker delta, to the lattice \( \Lambda_\mat{M} := \mat{M}^{-1}\mathbb Z^d \subset \mathbb Q^d\), we also get congruence classes. Let further \( \vec{e}_j := \bigl(\delta_{i,j})_{i=1}^d \) denote the \( j \)th unit vector. We obtain the pattern group \( \bigl(\mathcal P(\mat{M}),+) \) on the corresponding congruence classes \( [\vec{y}]_{{\mathbf E}_d} \), \( \vec{y}\in\Lambda_\mat{M} \), where the pattern \( \mathcal P(\mat{M}) \) is again any set of congruence class representants of the congruence classes on the lattice \( \Lambda_\mat{M} \). For any pattern \( \mathcal P(\mat{M}) \) we obtain a generating set by \( \mathcal G(\mat{M}) = \mat{M}\mathcal P(\mat{M}) \). Using a geometrical argument \cite[Lemma II.7]{dBHS:1993}, we get \( |\mathcal P(\mat{M})| = |\mathcal G(\mat{M})| = |\det\mat{M}| =: m \). A special pattern \( \mathcal P_{\text{S}}(\mat{M}) \) and its corresponding generating set \( \mathcal G_{\text{S}}(\mat{M}) \) are given by
		\[
			\mathcal P_{\text{S}}(\mat{M}) := \bigl[-\tfrac{1}{2},\tfrac{1}{2}\bigr)^d\cap\Lambda_\mat{M}
			\quad\text{ and }\quad
			\mathcal G_{\text{S}}(\mat{M}) := \mat{M}\bigl[-\tfrac{1}{2},\tfrac{1}{2}\bigr)^d\cap\mathbb Z^d
			\text{.}
		\]
		We will apply the usual addition, when performing an addition on the set of representatives, i.e., for \( \vec{x},\vec{y}\in\mathcal P(\mat{M}) \) the expression \( \vec{x}+\vec{y} \) is an abbreviation for choosing the unique element of \( [\vec{x}+\vec{y}]_{{\mathbf E}_d}\cap\mathcal P(\mat{M}) \). In fact, for any discrete group \( \mathcal G = (S,+\bmod1) \) with respect to addition modulo 1, there exists a matrix \( \mat{M} \), whose pattern \( \mathcal P(\mat{M}) \) coincides with the set \( S \) \cite[Theorem 1.8]{Bergmann:2013b}.
		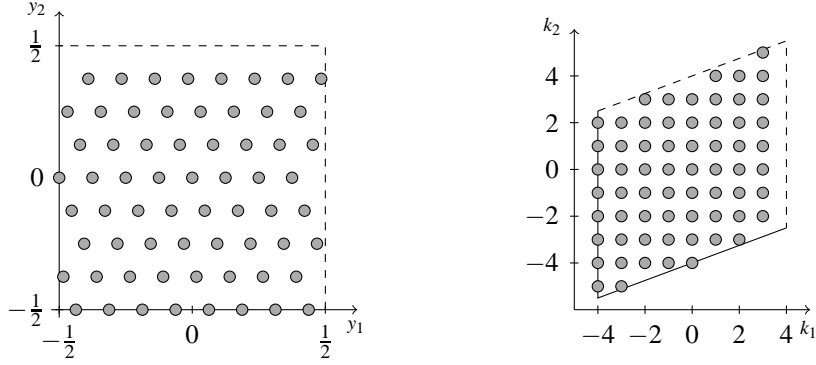
\begin{figure}[tbp]
		\centering
		\tikzstyle{every node}=[font=\footnotesize, label distance = 0pt]
		\begin{tikzpicture}[scale=3.5]
			\draw[thin, dashed] (.5,-.5) -- (.5,.5) -- (-.5,.5);
			\draw[->] (-.53cm,-.5cm) -- (.62cm,-.5cm) node[below] {$^{y_1}$};
			\draw[->] (-.5cm,-.53cm) -- (-.5cm,.63cm) node[left] {$^{y_2}$};
			\draw (0,-.5cm +.015cm) -- (0,-.5cm -.015cm) node[anchor=north] {$0$};
			\draw (-.5cm+.015cm,0) -- (-.5cm-.015cm,0) node[anchor=east] {$0$};
			\draw (-.5cm,-.5cm +.015cm) -- (-.5cm,-.5cm-.015cm) node[anchor=north] {$-\frac{1}{2}$};
			\draw (-.5cm+.015cm,-.5cm) -- (-.5cm-.015cm,-.5cm) node[anchor=east] {$-\frac{1}{2}$};
			\draw (.5cm,-.5cm+.015cm) -- (.5cm,-.5cm-.015cm) node[anchor=north] {$\frac{1}{2}$};
			\draw (-.5cm+.015cm,.5cm) -- (-.5cm-.015cm,.5cm) node[anchor=east] {$\frac{1}{2}$};
			\foreach \x/\y in {-0.5/0.,-0.484375/-0.375,-0.46875/0.25,-0.453125/-0.125,-0.4375/-0.5,-0.421875/0.125,-0.40625/-0.25,-0.390625/0.375,-0.375/0.,-0.359375/-0.375,-0.34375/0.25,-0.328125/-0.125,-0.3125/-0.5,-0.296875/0.125,-0.28125/-0.25,-0.265625/0.375,-0.25/0.,-0.234375/-0.375,-0.21875/0.25,-0.203125/-0.125,-0.1875/-0.5,-0.171875/0.125,-0.15625/-0.25,-0.140625/0.375,-0.125/0.,-0.109375/-0.375,-0.09375/0.25,-0.078125/-0.125,-0.0625/-0.5,-0.046875/0.125,-0.03125/-0.25,-0.015625/0.375,0./0.,0.015625/-0.375,0.03125/0.25,0.046875/-0.125,0.0625/-0.5,0.078125/0.125,0.09375/-0.25,0.109375/0.375,0.125/0.,0.140625/-0.375,0.15625/0.25,0.171875/-0.125,0.1875/-0.5,0.203125/0.125,0.21875/-0.25,0.234375/0.375,0.25/0.,0.265625/-0.375,0.28125/0.25,0.296875/-0.125,0.3125/-0.5,0.328125/0.125,0.34375/-0.25,0.359375/0.375,0.375/0.,0.390625/-0.375,0.40625/0.25,0.421875/-0.125,0.4375/-0.5,0.453125/0.125,0.46875/-0.25,0.484375/0.375}
				\node[circle,inner sep=1.5pt,fill=black!33,draw=black] at (\x,\y) {};
			\end{tikzpicture}\hspace{4em}
			\begin{tikzpicture}[scale=0.31]
				%
					\draw[->] (-5,-6) -- (5,-6) node[below] {$^{k_1}$};
					\draw[->] (-5,-6) -- (-5,6) node[left] {$^{k_2}$};
					\foreach \x in {-4,-2,0,2,4}
						\draw (\x,-6cm +.18cm) -- (\x,-6cm -.18cm) node[anchor=north] {$\x$};
					\foreach \y in {-4,-2,0,2,4}
					\draw (-5cm+.18cm,\y) -- (-5cm-.18cm,\y) node[anchor=east] {$\y$};
					\draw[thin, dashed] (-4.,2.5) -- (4.,5.5) -- (4.,-2.5);
					\draw[thin] (-4.,2.5) -- (-4.,-5.5) -- (4.,-2.5);
					\foreach \x/\y in {-4./-5.,-4./-4.,-4./-3.,-4./-2.,-4./-1.,-4./0.,-4./1.,-4./2.,-3./-5.,-3./-4.,-3./-3.,-3./-2.,-3./-1.,-3./0.,-3./1.,-3./2.,-2./-4.,-2./-3.,-2./-2.,-2./-1.,-2./0.,-2./1.,-2./2.,-2./3.,-1./-4.,-1./-3.,-1./-2.,-1./-1.,-1./0.,-1./1.,-1./2.,-1./3.,0./-4.,0./-3.,0./-2.,0./-1.,0./0.,0./1.,0./2.,0./3.,1./-3.,1./-2.,1./-1.,1./0.,1./1.,1./2.,1./3.,1./4.,2./-3.,2./-2.,2./-1.,2./0.,2./1.,2./2.,2./3.,2./4.,3./-2.,3./-1.,3./0.,3./1.,3./2.,3./3.,3./4.,3./5.}
					\node[circle,inner sep=1.5pt,fill=black!33,draw=black] at (\x,\y) {};
					\draw[use as bounding box,draw=none] (-5,-6-2.5) rectangle (5,6);
			\end{tikzpicture}
			\vspace{.5\baselineskip}
			\caption[]{The pattern \(\mathcal P_{\text{S}}(\mathbf{M})\) (left) and the generating set \(\mathcal G_{\text{S}}(\mathbf{M}^\T)\) (right), where \(\mathbf{M} = \bigl(\begin{smallmatrix}8&3\\0&8\end{smallmatrix}\bigr)\).}
			\label{fig:Pat}
		\end{figure}
		Figure \ref{fig:Pat} gives an example of a pattern \(\mathcal P(\mathbf{M})\) of a matrix \(\mathbf{M}\) and a generating set \(\mathcal G(\mathbf{M}^\T)\), where the matrix is an upper triangular matrix. By scaling and shearing, the points of the pattern lie dense along a certain direction.
	
		The discrete Fourier transform is defined by applying the Fourier matrix
		\[
		\mathcal F(\mat{M}) := \frac{1}{\sqrt{m}}\Bigl(\E^{-2\pi\I\vec{h}^\T\vec{y}}\Bigr)_{\vec{h}\in\mathcal G(\mat{M}^\T),\,\vec{y}\in\mathcal P(\mat{M})}
		\]
		to a vector \( \vec{a} = \bigl(a_{\vec{y}}\bigr)_{\vec{y}\in\mathcal P(\mat{M})}\in\mathbb C^m \) having the same order of elements as the columns of \(\mathcal F(\mat{M}) \). We write for the Fourier transform \( \vec{\hat{a}} := \bigl(\hat a_{\vec{h}}\bigr)_{\vec{h}\in\mathcal G(\mathbf{M}^\T)} = \sqrt{m}\mathcal F(\mat{M})\vec{a} \), where the vector \( \vec{\hat a} \) is ordered in the same way as the rows of the Fourier matrix \( \mathcal F(\mat{M}) \). Further investigations of patterns and generating sets, especially concerning subpatterns and shift invariant spaces, can be found in \cite{LangemannPrestin:2010}, which is extended in \cite{Bergmann:2013} with respect to bases and certain orderings of the elements of both sets to obtain a fast Fourier transform.
	
		Finally, we denote by \( \lambda_1(\mat{M}),\ldots,\lambda_d(\mat{M}) \) the eigenvalues of \( \mat{M} \) including their multiplicities in increasing order, i.e., for \( i<j \) we get \( |\lambda_i(\mat{M})|\leq|\lambda_j(\mat{M})|\).
	
		For the rest of this paper, let \( |\lambda_d(\mat{M})| \geq 2 \). To emphasize this fact, we call a matrix \( \mat{M} \) that fulfills \( |\lambda_d(\mat{M})|\geq2 \) an expanding matrix.
	%
	%
	\subsection{Function Spaces} 
		For functions \( f:\mathbb T^d\to\mathbb C\) on the
		torus \( \mathbb T^d:=\mathbb R^d/2\pi\mathbb Z^d \) consider the Banach spaces
		\(\operatorname{L}_p(\mathbb T^d)\), \( 1\leq p \leq \infty\), with norm
		\[
			\|f\,|\,\operatorname{L}_p(\mathbb T^d)\|^{p} 
			:= \frac{1}{(2\pi)^d}\int_{\mathbb T^d} |f(\vec{x})|^{p}\,\D\vec{x}
		\]
		and the usual modification for \( p=\infty \), that \( \|f|\operatorname{L}_{\infty}(\mathbb T^d)\| = \displaystyle\operatorname{ess\ sup}_{\vec{x}\in\mathbb T^d} |f(\vec{x})| \). Analogously for sequences
		\( \vec{c} := \{c_{\vec{z}}\}_{\vec{z}\in\mathcal X},\mathcal X \subseteq\mathbb Z^d \), the Banach spaces \( \ell_q(\mathcal X)\), \(1\leq q\leq\infty \), are defined with norm
		\[
		\|\vec{c}\,|\,\ell_q(\mathcal X)\|^q:=\sum_{\vec{k}\in\mathcal X}\bigl|c_{\vec{k}}\bigr|^q
		\text{,}
		\]
		again with the usual modification for \( q=\infty \). For \( f\in \operatorname{L}_1(\mathbb T^d) \) the Fourier coefficients are given by
		\[
			c_{\vec{k}}(f) := \frac{1}{(2\pi)^d}\int_{\mathbb T^d} f(\vec{x}) \E^{-\I\vec{k}^\T\vec{x}}\,\D\vec{x},\quad\vec{k}\in\mathbb Z^d\text{.}
		\]
		For \( \beta \geq 0 \), we define the ellipsoidal weight function \( \sigma^{\mat{M}}_{\beta} \), which was similarly introduced in \cite[Sec. 1.2]{Bergmann:2013b},
		\[
			\sigma^{\mat{M}}_{\beta}(\vec{k})
			:=
			\left(1+\|\mat{M}\|_2^2\|\mat{M}^{-\T}\vec{k}\|_2^2\right)^{\beta/2}
			,\quad\vec{k}\in\mathbb Z^d
			\text{,}
		\]
		to define for \( q \geq 1 \) the spaces
		\begin{equation*}
			A_{\mat{M},\,q}^{\beta}(\mathbb T^d) :=
			\left\{\left.
				f
				  \in \operatorname{L}_1(\mathbb T^d)
				\,
				\right|
				\bigl\|f\,\bigl|\,A_{\mat{M},q}^{\beta}\bigr\| < \infty
			\right\}\text{,}
		\end{equation*}
		where
		\begin{equation*}
			\bigl\|f\,\bigl|\,A_{\mat{M},q}^{\beta}\bigr\|
			:=
			\bigl\|
				\{\sigma_{\beta}^{\mat{M}}(\vec{k})
				c_{\vec{k}}(f)
				\}_{\vec{k}\in\mathbb Z^d}
			\,\bigr|\ell_q(\mathbb Z^d)
			\bigr\|\text{.}
		\end{equation*}
		A special case is given by \( A(\mathbb T^d) := A_{\mat{M},1}^0(\mathbb T^d) \), which is the Wiener algebra of all functions with absolutely convergent Fourier series. We see that \( \|\mat{M}^\T\|_2^2 = \lambda_d(\mat{M}^\T\mat{M}) > 1 \). For any diagonal matrix \( \mat{M} = \operatorname{diag}(N,\ldots,N) \), \( N\in\mathbb N \), the weight function simplifies to \( (1+\|\vec{k}\|_2^2)^{\beta/2} \) and these spaces resemble the spaces used in \cite{Sprengel:2000} to derive error bounds for interpolation by translates. Even more, if we fix \( \alpha\in\mathbb R \) and \( q\geq 1 \), due to the inequalities
		\begin{align*}
			\bigl(1+\|\mat{M}\|_2^2\|\mat{M}^{-\T}\vec{z}\|_2^2\bigr)^{\alpha/2}
			\leq
			\biggl(
			\frac{\lambda_d(\mat{M}^\T\mat{M})}{\lambda_1(\mat{M}^\T\mat{M})}
			\biggr)^{\alpha/2}
			\bigl(1+\|\vec{z}\|_2^2\bigr)^{\alpha/2}
		\end{align*}
		and
		\begin{align*}
			(1+\|\vec{z}\|_2^2)^{\alpha/2} &= (1+\|\mat{M}^\T\mat{M}^{-\T}\vec{z}\|_2^2)^{\alpha/2}
			\leq (1+\|\mat{M}\|_2^2\|\mat{M}^{-\T}\vec{z}\|_2^2)^{\alpha/2}\\
		\end{align*}
		we have, that all spaces \( A_{\mat{M},q}^\alpha \) of regular integer matrices \( \mat{M} \) are equal to \( A_{\mat{E}_d,q}^\alpha \), which is the same as \( A_{q}^\alpha \) in \cite{Sprengel:2000}. However, each of the different norms provides a different quantification of the functions \( f\in A_{q}^\alpha \). We keep the matrix \( \mat{M} \) in the notation of the space in order to distinguish the specific norm that we will use.

		For the weight \( \sigma^{\mat{M}}_\beta \) we finally have the following lemma.
		\begin{lemma}\label{lem:sigma-ineq}
			For a regular expanding matrix \(\mat{M}\in\mathbb Z^{d\times d}\), i.e., \( |\lambda_d(\mat{M})|\geq 2 \), and an ellipsoidal weight function \(\sigma_\beta^{\mat{M}}\), where \(\beta>0\) we have
			\begin{equation}\label{eq:weightsplit}
				\sigma_\beta^{\mat{M}}(\vec{k}+\mat{M}^\T\vec{z})
				\leq
				\|\mat{M}\|_2^{\beta}
				\sigma_\beta^{\mat{M}}(\vec{k})
				\sigma_\beta^{\mat{M}}(\vec{z}) \quad \text{ for }\vec{k},\vec{z}\in\mathbb Z^d
				\text{.}
			\end{equation}
		\end{lemma}
		\begin{proof}
			We have \( 2\leq\|\mat{M}\|_2 = \sqrt{\lambda_d(\mat{M}^\T\mat{M})} \). For \( \vec{z}=\vec{0} \) or \( \vec{k}=\vec{0} \) the assertion holds. For \( \vec{k},\vec{z}\neq\vec{0} \) we apply the triangle inequality and the submultiplicativity of the spectral norm to obtain
			\begin{align*}
				\sigma_\beta^{\mat{M}}(\vec{k}+\mat{M}^\T\vec{z})
				&=
					(1+\|\mat{M}\|_2^2\|\mat{M}^{-\T}(\vec{k}+\mat{M}^\T\vec{z})\|_2^2)^{\beta/2}\\
				&\leq
					\|\mat{M}\|_2^\beta(1+\|\mat{M}^{-\T}\vec{k}\|_2^2 + 2\|\mat{M}^{-\T}\vec{k}\|_2\|\vec{z}\|_2 + \|\vec{z}\|_2^2)^{\beta/2}\text{.}
			\end{align*}
				Using \( \|\mat{M}\|_2\|\mat{M}^{-\T}\vec{k}\|_2\geq 1 \), \( \|\vec{z}\|_2\geq 1 \) und \( \|\mat{M}\|_2\geq 2 \), we further get
			\begin{align*}
				\sigma_\beta^{\mat{M}}&(\vec{k}+\mat{M}^\T\vec{z})\\
				&\leq
				\|\mat{M}\|_2^\beta(1+\|\mat{M}^{-\T}\vec{k}\|_2^2 + \|\mat{M}\|_2^2\|\mat{M}^{-\T}\vec{k}\|_2^2\|\vec{z}\|_2^2 + \|\mat{M}^\T\mat{M}^{-\T}\vec{z}\|_2^2)^{\beta/2}\\
				&\leq
				\|\mat{M}\|_2^\beta
				(1+\|\mat{M}\|_2^2\|\mat{M}^{-\T}\vec{k}\|_2^2)^{\beta/2}
				(1+\|\mat{M}\|_2^2\|\mat{M}^{-\T}\vec{z}\|_2^2)^{\beta/2}\\
				&=
				\|\mat{M}\|_2^\beta
				\sigma_\beta^{\mat{M}}(\vec{k})
				\sigma_\beta^{\mat{M}}(\vec{z})\text{.}\tag*{\qed}
			\end{align*}
			\let\qed\relax
		\end{proof}
		\begin{remark}
			In the same way one would obtain
			\( 	\sigma_\beta^{\mat{M}}(\vec{k}+\mat{M}^\T\vec{z})
			\leq2^\beta
			\|\mat{M}\|_2^{\beta}
			\sigma_\beta^{\mat{M}}(\vec{k})
			\sigma_\beta^{\mat{M}}(\vec{z}) \)
			for all regular integral matrices with \( \|\mat{M}\|_2\geq 1  \) with slightly bigger constant \( 2^\beta \).
			For the matrices of interest, this slight difference is not that important and we will focus on the former one for simplicity.
		\end{remark}
%
%
%
\section{Interpolation and the Strang-Fix Condition}\label{sec:IPSFC} 
	This section is devoted to interpolation on a pattern \( \mathcal P(\mat{M}) \) and its corresponding periodic Strang-Fix conditions. The periodic Strang-Fix conditions were introduced in \cite{Brumme:1994, Poe95} for tensor product grids as a counterpart to the strong Strang-Fix conditions on the Euclidean space \( \mathbb R^d \) and generalized in \cite{Sprengel:1997, Sprengel:2000}. We generalize them to arbitrary patterns on the torus.
	
	A space of functions \( V \) is called \( \mat{M} \)-invariant, if for all \( \vec{y}\in\mathcal P(\mat{M}) \) and all functions \( \varphi\in V \) the translates \( \operatorname{T}_{\vec{y}}\!\varphi := \varphi(\circ-2\pi\vec{y})\in V\). Especially the space
	\[
		V_{\mat{M}}^\varphi := \operatorname{span}\bigr\{\operatorname{T}_{\vec{y}}\!\varphi\,;\,\vec{y}\in\mathcal P(\mat{M})\bigl\}
	\]
	of translates of \( \varphi \) is \( \mat{M} \)-invariant. A function \( \xi\in V_{\mat{M}}^\varphi \) is of the form \( \xi=\!\displaystyle\sum_{\vec{y}\in\mathcal P(\mat{M})} a_{\vec{y}}\operatorname{T}_{\vec{y}}\!\varphi\).
	For \( \varphi\in \operatorname{L}_1(\mathbb T^d) \) an easy calculation on the Fourier coefficients using the unique decomposition of \( \vec{k}\in\mathbb Z^d \) into \( \vec{k} = \vec{h}+\mat{M}^\T\vec{z} \), \( \vec{h}\in\mathcal G(\mat{M}^\T), \vec{z}\in\mathbb Z^d \), yields, that \( \xi\in V_{\mat{M}}^\varphi \) holds if and only if \begin{equation}\label{eq:inTranslates:ck}
		c_{\vec{h}+\mat{M}^\T\vec{z}}(\xi) = \hat a_{\vec{h}}c_{\vec{h}+\mat{M}^\T\vec{z}}(\varphi) \quad\text{for all } \vec{h}\in\mathcal G(\mat{M}^\T), \vec{z}\in\mathbb Z^d\text{,}
	\end{equation}
	is fulfilled, where \( \vec{\hat a} = \bigl(\hat a_{\vec{h}}\bigr)_{\vec{h}\in\mathcal G(\mathbf{M}^\T)} = \sqrt{m}\mathcal F(\mat{M})\vec{a} \) denotes the discrete Fourier transform of \( \vec{a}\in\mathbb C^m \).
	
	Further, the space of trigonometric polynomials on the generating set \( \mathcal G_{\text{S}}(\mat{M}^\T) \) is denoted by
	\[
	\mathcal T_{\mat{M}} := \Bigl\{\varphi\,;\,
	\varphi=
	\sum_{\vec{h}\in\mathcal G_{\text{S}}(\mat{M}^\T)}a_{\vec{h}}\E^{\I\vec{h}^\T\circ},\ a_{\vec{h}}\in\mathbb C
	\Bigr\}\text{.}
	\]
	For any function \( \varphi\in \operatorname{L}_1(\mathbb T^d) \) the Fourier partial sum \( \operatorname{S}_{\mat{M}}\varphi\in\mathcal T_{\mat{M}} \) given by
	\[
	\operatorname{S}_{\mat{M}}\varphi := \sum_{\vec{h}\in\mathcal G_{\text{S}}(\mat{M}^\T)} c_{\vec{h}}(\varphi)\E^{\I\vec{h}^\T\circ}
	\]
	is such a trigonometric polynomial.
	
	The discrete Fourier coefficients of a pointwise given \( \varphi \) are defined by
	\begin{equation}\label{eq:discrete-Fourier}
		c_{\vec{h}}^{\mat{M}}(\varphi) := \frac{1}{m}\sum_{\vec{y}\in\mathcal P(\mat{M})} \varphi(2\pi\vec{y})\E^{-2\pi\I\vec{h}^\T\vec{y}}
		,\quad\vec{h}\in\mathcal G(\mat{M}^\T)
		\text{,}
	\end{equation}
	which are related to the Fourier coefficients for \( \varphi\in A(\mathbb T^d) \) by the following aliasing formula.
	\begin{lemma}\label{lem:Aliasing}
		Let \( \varphi\in A(\mathbb T^d) \) and the regular matrix \(\mat{M}\in\mathbb Z^{d\times d}\) be given. Then the discrete Fourier coefficients \( c_{\vec{h}}^{\mat{M}}(\varphi) \) are given by
		\begin{equation}\label{eq:aliasing}
			c_{\vec{k}}^{\mat{M}}(\varphi) = \sum_{\vec{z}\in\mathbb Z^d} c_{\vec{k}+\mat{M}^\T\vec{z}}(\varphi),\quad \vec{k}\in\mathbb Z^d\text{.}
		\end{equation}
	\end{lemma}
	\begin{proof}
		Writing each point evaluation of \( \varphi \) in \eqref{eq:discrete-Fourier} as its Fourier series, we obtain due to the absolute convergence of the series
		\begin{align*}
			c_{\vec{k}}^{\mat{M}}(\varphi)
			&=
			\frac{1}{m}\sum_{\vec{y}\in\mathcal P(\mat{M})}
			\left(
				\sum_{\vec{h}\in\mathbb Z^d}
				c_{\vec{h}}(\varphi)\E^{2\pi\I\vec{h}^\T\vec{y}}
			\right)
			\E^{-2\pi\I\vec{k}^\T\vec{y}}\\
			&=
			\frac{1}{m}
			\sum_{\vec{h}\in\mathbb Z^d}
			c_{\vec{h}}(\varphi)
			\sum_{\vec{y}\in\mathcal P(\mat{M})}
				\E^{-2\pi\I(\vec{k}-\vec{h})^\T\vec{y}}\\
			&=
			\sum_{\vec{z}\in\mathbb Z^d}
			c_{\vec{k}+\mat{M}^\T\vec{z}}(f)\text{.}
		\end{align*}
		The last equality is valid because the sum over \( \vec{y} \) simplifies to \( m \) if
		\( \vec{k}\equiv\vec{h}\bmod\mat{M}^\T \), and vanishes otherwise, cf. \cite[Lemma 2.7]{SloanJoa:1994}.
	\end{proof}
	\begin{definition}
		\label{def:IP}
		Let \( \mat{M}\in\mathbb Z^{d\times d} \) be a regular matrix. A function \( \operatorname{I}_\mat{M} \in V_\mat{M}^\varphi\) is called fundamental interpolant or Lagrange function of \(  V_\mat{M}^\varphi \) if
		\begin{equation*}
			\operatorname{I}_\mat{M}(2\pi\vec{y}) := \delta_{\vec{0},\vec{y}}^{{\mathbf E}_d},\quad \vec{y}\in\mathcal P(\mat{M}),\quad\text{where }
			\delta_{\vec{x},\vec{y}}^{{\mathbf M}} :=
			\begin{cases}
				1 &\mbox{ if } \vec{y}\equiv \vec{x}\bmod{\mathbf M},\\
				0 &\mbox{ else.}
			\end{cases}
		\end{equation*}
	\end{definition}
	The following lemma characterizes the existence of such a fundamental interpolant.
	\begin{lemma}\label{lem:FI:Existenz}
		Given a regular matrix \(\mat{M}\in\mathbb Z^{d\times d}\) and a function \(\varphi\in A(\mathbb T^d)\), the fundamental interpolant \(\operatorname{I}_\mat{M}\in V_\mat{M}^\varphi\) exists if and only if
		\begin{equation}\label{eq:lem:Existenz}
			\sum_{\vec{z}\in\mathbb Z^d} c_{\vec{h}+\mat{M}^{\T}\vec{z}}(\varphi) \neq 0,\quad\text{ for all }\vec{h}\in\mathcal G(\mat{M}^{\T}).
		\end{equation}
		If the fundamental interpolant \( \operatorname{I}_\mat{M}\in V_\mat{M}^\varphi \) exists, it is uniquely determined.
	\end{lemma}
	\begin{proof}
		Assume the fundamental interpolant \( \operatorname{I}_{\mat{M}}\in V_{\mat{M}}^\varphi\) exists. Hence, there exists a vector \( \vec{a} = (a_{\vec{h}})_{\vec{h}\in\mathcal G(\mat{M}^{\T})} \) such that for its Fourier transform \( \vec{\hat a} = \sqrt{m}\mathcal F(\mat{M})\vec{a} \) it holds due to~\eqref{eq:inTranslates:ck} that
		\begin{equation*}
			c_{\vec{h}+\mat{M}^{\T}\vec{z}}(\operatorname{I}_\mat{M}) = \hat a_{\vec{h}}c_{\vec{h}+\mat{M}^{\T}\vec{z}}(\varphi),\quad\vec{h}\in\mathcal G(\mat{M}^{\T}),\ \vec{z}\in\mathbb Z^d\text{.}
		\end{equation*}
		Applying this equality to the discrete Fourier coefficients of \( \operatorname{I}_{\mat{M}} \) yields
		\begin{equation}\label{eq:interpolant-BS}
			c_{\vec{h}}^\mat{M}(\operatorname{I}_\mat{M})
			=
			\sum_{\vec{z}\in\mathbb Z^d} c_{\vec{h}+\mat{M}^{\T}\vec{z}}(\operatorname{I}_\mat{M})
			=
			\hat a_{\vec{h}}\sum_{\vec{z}\in\mathbb Z^d} c_{\vec{h}+\mat{M}^{\T}\vec{z}}(\varphi)
			=
			\hat a_{\vec{h}}c_{\vec{h}}^\mat{M}(\varphi)\text{.}
		\end{equation}
		The discrete Fourier coefficients are known by Definition~\ref{def:IP} and \cite[Lemma 2.7]{SloanJoa:1994} as \(c_{\vec{h}}^\mat{M}(\operatorname{I}_\mat{M}) = \tfrac{1}{m}\), \( \vec{h}\in\mathcal G(\mat{M}^\T) \), which is nonzero for all \( \vec{h} \) and hence \eqref{eq:lem:Existenz} follows.
		
		On the other hand, if \eqref{eq:lem:Existenz} is fulfilled, then the function \( \xi \), which is defined by
		\begin{equation}\label{eq:ck-Interpolant}
			c_{\vec{k}}(\xi) = \frac{c_{\vec{k}}(\varphi)}{m c_{\vec{k}}^\mat{M}(\varphi)},\quad \vec{k}\in\mathbb Z^d\text{,}
		\end{equation}
		is in the space \( V_{\mat{M}}^\varphi \) having the coefficients \( \hat a_{\vec{h}} = (mc_{\vec{h}}^\mat{M}(\varphi))^{-1}\), \( \vec{h}\in\mathcal G(\mat{M}^\T) \). The discrete Fourier coefficients also fulfill \( c_{\vec{h}}^{\mat{M}}(\xi) = \tfrac{1}{m} \). Hence, again by Definition~\ref{def:IP} and \cite[Lemma 2.7]{SloanJoa:1994}, \( \xi \) is a fundamental interpolant with respect to the pattern \( \mathcal P(\mat{M}) \). If the fundamental interpolant \( \operatorname{I}_\mat{M} \) exists, \eqref{eq:ck-Interpolant} also provides uniqueness.
	\end{proof}
	The associated interpolation operator \( \operatorname{L}_{\mat{M}}f \) is given by
	\begin{equation}\label{eq:interpolantck}	
	\operatorname{L}_{\mat{M}}f := \sum_{\vec{y}\in\mathcal P(\mat{M})} f(2\pi\vec{y})\operatorname{T}_{\vec{y}}\operatorname{I}_{\mat{M}}
	=
	m
	\sum_{\vec{k}\in\mathbb Z^d}
	c_{\vec{k}}^{\mat{M}}(f)c_{\vec{k}}(\operatorname{I}_{\mat{M}})
	\E^{\I\vec{k}^\T\circ}
	\text{.}
	\end{equation}
	The following definition introduces the periodic Strang-Fix conditions, which require the Fourier coefficients \( c_{\vec{k}}(\operatorname{I}_{\mat{M}}) \) of the fundamental interpolant to decay in a certain ellipsoidal way. The condition number \( \kappa_{\mat{M}} \) of \( \mat{M} \) is given by
	\[
		 \kappa_{\mat{M}}
		:=
		\sqrt{\frac{\lambda_d(\mat{M}^\T\mat{M})}{\lambda_1(\mat{M}^\T\mat{M})}}
		= \|\mat{M}\|_2\|\mat{M}^{-1}\|_2
		\text{.}
	\]
%
%
%
	\begin{definition}\label{def:SFC}
		Given a regular expanding matrix \( \mat{M}\in\mathbb Z^{d\times d} \), a fundamental interpolant \( \operatorname{I}_\mat{M}\in \operatorname{L}_1(\mathbb T^d) \) fulfills the  ellipsoidal (periodic) Strang-Fix conditions of order \( s > 0 \) for \( q \geq 1 \)  and an \( \alpha \in \mathbb R^+ \),
		if there exists a nonnegative sequence \( \vec{b}=\{b_{\vec{z}}\}_{\vec{z}\in\mathbb Z^d}\subset \mathbb R_0^+ \), such that for all \( \vec{h}\in\mathcal G (\mat{M}^\T) \), \( \vec{z}\in\mathbb Z^d\backslash\{\vec{0}\} \) we have
		\vspace{.5\baselineskip}
		\begin{enumerate}
			\setlength{\itemsep}{\baselineskip}
			\item \(
			|1-mc_{\vec{h}}(\operatorname{I}_\mat{M})|
			\leq b_{\vec{0}}
			\kappa_{\mat{M}}^{-s}
			\|\mat{M}^{-\T}\vec{h}\|_2^s
			\text{,}
			\)
			\label{item:SFC-inner}
			\item \( |mc_{\vec{h}+\mat{M}^\T\vec{z}}(\operatorname{I}_\mat{M})| \leq
			b_{\vec{z}}\kappa_{\mat{M}}^{-s}\|\mathbf{M}\|_2^{-\alpha}\|\mat{M}^{-\T}\vec{h}\|_2^{s} \)\label{item:SFC-outer}
		\end{enumerate}
		\vspace{.5\baselineskip}
		with\[ \gamma_{\mathrm{SF}} := \| \{
			\sigma_\alpha^{\mat{M}}(\vec{z})
			b_{\vec{z}}
			\}_{\vec{z}\in\mathbb Z^d} | \ell_q(\mathbb Z^d)\| < \infty
			\text{.}
		\]
	\end{definition}
	For both properties we enforce a stronger decay than by the ellipse defined by the level curves of \( \|\mat{M}^{-\T}\circ\|_2 \), i.e., we have an upper bound by \( \kappa_{\mat{M}}^{-s}\|\mat{M}^{-\T}\vec{h}\|_2^s \leq \kappa_{\mat{M}}^{-s}\|\mat{M}^{-\T}\|_2^s\|\vec{h}\|_2^s = \bigl(\lambda_d(\mat{M}^\T\mat{M})\bigr)^{-s/2}\|\vec{h}\|^s \). The second one enforces a further stronger decay with respect to \( \alpha \), i.e., \( \kappa_{\mat{M}}^{-s}\|\mathbf{M}\|_2^{-\alpha}\|\mat{M}^{-\T}\vec{h}\|_2^{s} \leq (\lambda_d(\mat{M}^\T\mat{M}))^{-(\alpha+s)/2}\|\vec{h}\|_2^{s} \). For the one-dimensional case or the tensor product case, i.e., \( \mat{M} = \operatorname{diag}(N,\ldots,N) \) we have \( \kappa_{\mat{M}}=1 \), \( \lambda_1(\mat{M}^\T\mat{M})=\lambda_d(\mat{M}^\T\mat{M})=N \), and this simplifies to the already known case \( N^{-\alpha-s}\|\vec{h}\|_2 \). Looking at the level curves of the map \( \|\mat{M}^{-\T}\circ\| \), we see they produce ellipsoids, where \( |\lambda_d(\mat{M})| \) is the length of the longest axis. Hence the decay is normalized with respect to the longest axis of the ellipsoid.
%
%
%
\section{Error Bounds for Interpolation}\label{sec:IPErr} 
	In order to investigate the error of interpolation \( \|f-\operatorname{L}_{\mat{M}}f\| \), where \( \operatorname{L}_{\mat{M}} \) is the interpolation operator into \( V_{\mat{M}}^{\varphi} \) for certain \( \varphi\in A(\mathbb T^d) \), we use the triangle inequality with respect to any norm
	\begin{equation*}
		\|f-\operatorname{L}_{\mat{M}}f \|
\leq
		\|\operatorname{S}_{\mat{M}}f-\operatorname{L}_{\mat{M}}\operatorname{S}_{\mat{M}}f\|
		+ \|f-\operatorname{S}_{\mat{M}}f \|
		+ \|\operatorname{L}_{\mat{M}}(f-\operatorname{S}_{\mat{M}}f)\|
	\end{equation*}
	and look at these three terms separately.
	\begin{theorem}\label{theorem:trigapprox}
		For an expanding regular matrix \( \mat{M} \in\mathbb Z^{d\times d}\), a trigonometric polynomial \( f\in\mathcal T_{\mat{M}} \) and a fundamental interpolant \( \operatorname{I}_{\mat{M}}\in A(\mathbb T^d) \) fulfilling the ellipsoidal Strang-Fix conditions for fixed values \( s\geq0\), \(\alpha>0 \), and \( q \geq 1 \) we have
		\begin{equation}\label{eq:trigapprox}
			\left\|\left.
					f - \operatorname{L}_{\mat{M}}f
			\right|A_{\mat{M},\,q}^{\alpha}
			\right\|
			\leq
			\left(
				\frac{1
				}{\|\mat{M}\|_2}
			\right)^{\!s}
			\gamma_{\mathrm{SF}}
			\left\|\left. f
			\right|A_{\mat{M},\,q}^{\alpha+s}
			\right\|\text{.}
		\end{equation}
	\end{theorem}
	\begin{proof}
		The proof is given for \( q<\infty \). For \( q=\infty \) the same arguments apply with the usual modifications with respect to the norm.
		Looking at the Fourier coefficients of \( \operatorname{L}_{\mat{M}}f \) in \eqref{eq:interpolantck} for \(f\in\mathcal T_{\mat{M}}\) yields
		\[
			c_{\vec{h}}(\operatorname{L}_{\mat{M}}f)
			= mc_{\vec{h}}^{\mat{M}}(f)c_{\vec{h}}(\operatorname{I}_{\mat{M}})
			= mc_{\vec{h}}(f)c_{\vec{h}}(\operatorname{I}_{\mat{M}}),
			\quad \vec{h}\in\mathcal G(\mat{M}^\T)\text{,}
		\]
		and hence we have
		\begin{equation*}
			f - \operatorname{L}_{\mat{M}}f
			=
				\sum_{\vec{k}\in\mathbb Z^d}\bigl(
					c_{\vec{k}}(f) - mc_{\vec{k}}^{\mat{M}}(f)c_{\vec{k}}(\operatorname{I}_{\mat{M}})
				\bigr)
				\E^{\I\vec{k}^\T\circ}
			\text{.}
		\end{equation*}
		Using the unique decomposition of \( \vec{k}\in\mathbb Z^d \) into \( \vec{k} = \vec{h} + \mat{M}^\T\vec{z}\), \(\vec{h}\in\mathcal G(\mat{M}^\T),\ \vec{z}\in\mathbb Z^d \), yields
		\begin{equation*}
						f - \operatorname{L}_{\mat{M}}f
			=
				\sum_{\vec{h}\in\mathcal G(\mat{M}^\T)}
				c_{\vec{h}}(f)e^{\I\vec{h}^\T\circ}
				\Bigl(\!
					\bigl(
					1-mc_{\vec{h}}(\operatorname{I}_{\mat{M}})
					\bigr) -
					\sum_{\vec{z}\in\mathbb Z^d\backslash\{\vec{0}\}}
					m c_{\vec{h}+\mat{M}^\T\vec{z}}(\operatorname{I}_{\mat{M}})
					\E^{\I\mat{M}^\T\vec{z}\circ}
				\Bigr)
				\text{.}
		\end{equation*}
		Applying the definition of the norm in \( A_{\mat{M},\,q}^{\alpha}(\mathbb T^d) \), we obtain
		\begin{equation*}
			\begin{split}
					&\Bigl\|
							f - \operatorname{L}_{\mat{M}}f
					\Bigr|A_{\mat{M},\,q}^{\alpha}
					\Bigr\|^q\\
				&\ \ =
				\!\!\!\!
				\sum_{\vec{h}\in\mathcal G(\mat{M}^\T)}
				\!\!\!\!
				|c_{\vec{h}}(f)|^q
				\Bigl(
					| (1-mc_{\vec{h}}(\operatorname{I}_{\mat{M}}))\sigma_{\alpha}^{\mat{M}}(\vec{h}) |^q
					+
					\!\!\!\!
					\sum_{\vec{z}\in\mathbb Z^d\backslash\{\vec{0}\}}
					\!\!\!\!
					| m c_{\vec{h}+\mat{M}^\T\vec{z}}(\operatorname{I}_{\mat{M}}) \sigma_{\alpha}^{\mat{M}}(\vec{h}+\mat{M}^\T\vec{z}) | ^q
					\!
				\Bigr)\text{.}
			\end{split}
		\end{equation*}
		Using the Strang-Fix conditions of the fundamental interpolant \(\operatorname{I}_{\mat{M}}\) and Lemma~\ref{lem:sigma-ineq} we get the following upper bound
		\begin{align*}
			\Bigl\|
			f - \operatorname{L}_{\mat{M}}f
			\Bigr|&A_{\mat{M},\,q}^{\alpha}
			\Bigr\|^q\\
			&\leq
			\sum_{\vec{h}\in\mathcal G(\mat{M}^\T)}
			|c_{\vec{h}}(f)|^q
			\biggl(
				b_{\vec{0}}^q\|\mat{M}^{-\T}\vec{h}\|^{sq}_2\sigma_{\alpha q}^\mat{M}(\vec{h})\kappa_{\mat{M}}^{-sq}
				\\&\qquad+
				\sum_{\vec{z}\in\mathbb Z^d\backslash\{\vec{0}\}}
				b_{\vec{z}}^q\kappa_{\mat{M}}^{-sq}\|\mathbf{M}\|_2^{-\alpha q}\|\mat{M}^{-\T}\vec{h}\|_2^{sq}\sigma_{\alpha q}^{\mat{M}}(\vec{h}+\mat{M}^\T\vec{z})
			\biggr)
				\\
				&\leq
				\biggl(
				\sum_{\vec{h}\in\mathcal G(\mat{M}^\T)}
					|c_{\vec{h}}(f)|^q
					\|\mat{M}^{-\T}\|_2^{sq}\kappa_{\mat{M}}^{-sq}\|\vec{h}\|_2^{sq}\sigma_{\alpha q}^{\mat{M}}(\vec{h})
				\biggr)
				\biggl(
					\sum_{\vec{z}\in\mathbb Z^d}
					(\sigma_\alpha^{\mat{M}}(\vec{z})b_{\vec{z}})^q
				\biggr)\\
				&\leq
				\gamma_{\mathrm{SF}}^q
				\|\mat{M}\|_2^{-sq}
				\bigl\|f\bigr|A^{\alpha+s}_{\mat{M},q}\bigr\|^q\text{.}\tag*{\qed}
		\end{align*}
		\let\qed\relax
	\end{proof}
%
%
	\begin{theorem}\label{theorem:approxtrig}
		Let \( \mat{M}\in\mathbb Z^{d\times d} \) be regular. If \(f\in A_{\mat{M},\,q}^{\mu}(\mathbb T^d)\), \(q\geq1\), \(\mu\geq\alpha\geq 0\), then
		\begin{equation*}
			\left\|\left.f-\operatorname{S}_{\mat{M}}f\right|A_{\mat{M},\,q}^{\alpha}
			\right\|
			\leq
			\left(\frac{2}{\|\mat{M}\|_2}\right)^{\mu-\alpha}
				\|\left.f\right| A_{\mat{M},\,q}^{\mu}
				\|\text{.}
		\end{equation*}
	\end{theorem}
	\begin{proof}
		This proof is given for \( q<\infty \). For \( q=\infty \) the same arguments apply with the usual modifications with respect to the norm. We examine the left-hand side of the inequality, apply \(\sigma_{\alpha}^{\mat{M}}(\vec{k}) = \sigma_{\alpha-\mu}^{\mat{M}}(\vec{k})\sigma_{\mu}^{\mat{M}}(\vec{k})\), and obtain
		\begin{equation*}
			\begin{split}
				\bigl\| f - \operatorname{S}_{\mat{M}}f\,\bigr|\,A_{\mat{M},q}^{\alpha} \bigr\|
			&=
				\bigl\|
				 	\{ \sigma_{\alpha}^{\mat{M}}(\vec{k})c_{\vec{k}}(f)\}_{\vec{k}\in\mathbb Z^d \backslash \mathcal G_{\text{S}}(\mat{M}^\T)}
				\,\bigr|\,
				\ell_q\bigl(\mathbb Z^d\backslash\mathcal G_{\text{S}}(\mat{M}^\T)\bigr)\bigr\|\\
			&\leq
				\max_{\vec{k}\in\mathbb Z^d\backslash\mathcal G_{\text{S}}(\mat{M}^\T)}\sigma_{\alpha-\mu}^{\mat{M}}(\vec{k})
				\bigl\| f\,\bigr|\, A_{\mat{M},q}^{\mu}
				\bigr\|\text{.}
			\end{split}
		\end{equation*}
		The decomposition of \( \vec{k}\in\mathbb Z^d\backslash\mathcal G_{\text{S}}(\mat{M}^\T) \) into \(\vec{k} = \vec{h} + \mat{M}^\T\vec{z}\), \(\vec{h}\in\mathcal G_{\text{S}}(\mat{M}^\T)\), yields that \(\vec{0}\neq\vec{z}\in\mathbb Z^d\) and hence none of these integral points lies inside the parallelotope \( \mat{M}^\T\bigl[-\tfrac{1}{2},\tfrac{1}{2}\bigr)^d \).
		Hence, \( \mat{M}^{-\T}\vec{k} \) lies outside \( \bigl[-\tfrac{1}{2},\tfrac{1}{2}\bigr)^d \) and we have
		\begin{align*}
				\max_{\vec{k}\in\mathbb Z^d\backslash\mathcal G_{\text{S}}(\mat{M}^\T)} \sigma_{\alpha-\mu}^{\mat{M}}(\vec{k})
				&=
				\max_{\vec{k}\in\mathbb Z^d\backslash\mathcal G_{\text{S}}(\mat{M}^\T)}
					(1+\|\mat{M}\|_2^2\|\mat{M}^{-\T}\vec{k}\|_2^2)^{\frac{\alpha-\mu}{2}}\\
				&\leq
					\max_{j\in\{1,\ldots,d\}}
			\bigg(1+\frac{\|\mat{M}\|_2^2}{4}
			\bigg)^{\frac{\alpha-\mu}{2}}\\
				&\leq
			\bigg(\frac{\|\mat{M}\|_2^2}{4
			}
			\bigg)^{\frac{\alpha-\mu}{2}}
			\text{.}\tag*{\qed}
		\end{align*}
		\let\qed\relax
	\end{proof}
	Indeed, Theorem \ref{theorem:approxtrig} does hold for any regular matrix \( \mat{M} \). It is not required that the matrix has to be expanding. For the following theorem, let \( |\vec{z}| := \bigl(|z_1|,\ldots,|z_d|\bigr)^\T\) denote the vector of the absolute values of the elements of the vector \( \vec{z}\in\mathbb Z^d \).
	\begin{theorem}\label{theorem:approxinterpol}
		For an expanding regular matrix \( \mat{M}\in\mathbb Z^{d\times d} \) let \( \operatorname{I}_{\mat{M}} \) be a fundamental interpolant such that
		\begin{equation*}
			\gamma_{\mathrm{IP}} := m\begin{cases}\!
				\max\limits_{\vec{h}\in\mathcal G_{\text{S}}(\mat{M}^\T)}
					\biggl(
						|c_{\vec{h}}(\operatorname{I}_{\mat{M}})|^q +
							\|\mat{M}\|_2^{\alpha q}
							\displaystyle\sum\limits_{\vec{z}\in\mathbb Z^d\backslash\{\vec{0}\}}
							 |\sigma_{\alpha}^{\mat{M}}(\vec{z})c_{\vec{h}+\mat{M}^\T\vec{z}}(\operatorname{I}_{\mat{M}})|^q\!
					\biggr)^{\!1/q}\\
			\hspace{.7\textwidth}
			\mbox{ if } q < \infty\text{,}
			\\
			\ \\
			\max\limits_{\vec{h}\in\mathcal G_{\text{S}}(\mat{M}^\T)}
				\sup
				\Bigl\{ |c_{\vec{h}}(\operatorname{I}_{\mat{M}})|,\
					\|\mat{M}\|_2^{\alpha}
					|\sigma_{\alpha}^{\mat{M}}(\vec{z})c_{\vec{h}+\mat{M}^\T\vec{z}}(\operatorname{I}_{\mat{M}})|
					; \vec{z}\in\mathbb Z^d\backslash \{\vec{0}\}
				\Bigr\}
			\\
			\hspace{.7\textwidth}
			\mbox{ if } q=\infty
			\end{cases}
		\end{equation*}
		is finite. Then we get for \(f\in A^{\mu}_{\mat{M},\,q}(\mathbb T^d)\), \(q \geq 1\), \(\mu\geq\alpha\geq 0\), and \(\mu > d(1-1/q)\)
		\begin{equation*}
			\left\|\left.
			\operatorname{L}_{\mat{M}}(f-\operatorname{S}_{\mat{M}}f)
			\right|
			A_{\mat{M},\,q}^{\alpha}
			\right\|
			\leq
			\gamma_{\mathrm{IP}}\gamma_{\mathrm{Sm}}
			\left(\frac{1}{\|\mat{M}\|_2}\right)^{\mu-\alpha}
			\bigl\|f\bigr| A_{\mat{M},\,q}^{\mu}\bigr\|\text{,}
		\end{equation*}
		where
		\begin{equation*}
			\gamma_{\mathrm{Sm}} := (1+d)^{\alpha/2}2^{-\mu}
			\begin{cases}
				\biggl(
					\sum\limits_{\vec{z}\in\mathbb Z^d\backslash\{\vec{0}\}}
					\|2|\vec{z}|-\vec{1}\|^{-p\mu}_2
					\biggr)^{1/p}
					&\mbox{ if } q > 1, \tfrac{1}{p}+\tfrac{1}{q} = 1\text{,}\\
					\sup\limits_{\vec{z}\in\mathbb Z^d\backslash\{\vec{0}\}}\|2|\vec{z}|-\vec{1}\|^{-\mu}_2
					&\mbox{ if } q = 1\text{.}
			\end{cases}
		\end{equation*}
	\end{theorem}
	\begin{proof}
		This proof is given for \( q<\infty \). For \( q=\infty \) the same arguments apply with the usual modifications with respect to the norm. We write the norm on the left-hand side of the inequality as
		\begin{equation*}
			\begin{split}
				\bigl\|\operatorname{L}_{\mat{M}}(f-\operatorname{S}_{\mat{M}}f)
				\bigr|A_{\mat{M},\,q}^{\alpha}\bigr\|^q
				&=
				\Bigl\|\sum_{\vec{k}\in\mathbb Z^d}\sigma^{\mat{M}}_{\alpha}(\vec{k})c_{\vec{k}}\left(\operatorname{L}_{\mat{M}}(f-\operatorname{S}_{\mat{M}}f)
				\right)e^{\I\vec{k}^\T\circ}
				\Bigr|
				A_{\mat{M},\,q}^{\alpha}
				\Bigr\|^q\\
				&=
				\sum_{\vec{k}\in\mathbb Z^d}
				\left|
				\sigma^{\mat{M}}_{\alpha}(\vec{k}) m
				 c_{\vec{k}}^{\mat{M}}(f-\operatorname{S}_{\mat{M}}f)c_{\vec{k}}(\operatorname{I}_{\mat{M}})
				\right|^q\text{.}
			\end{split}
		\end{equation*}
By decomposing \( \vec{k} = \vec{h}+\mat{M}^\T\vec{z}\), \(\vec{h}\in\mathcal G_{\text{S}}(\mat{M}^\T)\), \(\vec{z}\in\mathbb Z^d  \), and using Lemma~\ref{lem:sigma-ineq} we obtain
		\begin{equation*}
			\begin{split}
				\Bigl\|\operatorname{L}_{\mat{M}}(f-\operatorname{S}_{\mat{M}}f)&\Bigr|
				A_{\mat{M},\,q}^{\alpha}
				\Bigr\|^q\\
				&\leq
				\sum_{\vec{h}\in\mathcal G_{\text{S}}(\mat{M}^\T)} \Bigl|
				\sigma^{\mat{M}}_{\alpha}(\vec{h})mc_{\vec{h}}^{\mat{M}}(f-\operatorname{S}_{\mat{M}}f)
				\Bigr|^q
				\\
				&\hspace{4em}
				\times
				\biggl(
				|c_{\vec{h}}(\operatorname{I}_{\mat{M}})|^q
				+
				\|\mat{M}\|_2^{\alpha q}
				\sum\limits_{\vec{z}\in\mathbb Z^d\backslash\{\vec{0}\}}
				| \sigma^{\mat{M}}_{\alpha}(\vec{z})c_{\vec{h}+\mat{M}^\T\vec{z}}(\operatorname{I}_{\mat{M}}) |^q
				\biggr)\\
				&\leq
				\gamma_{\mathrm{IP}}^q\sum\limits_{\vec{h}\in\mathcal G_{\text{S}}(\mat{M}^\T)}
					| \sigma_{\alpha}^{\mat{M}}(\vec{h})c_{\vec{h}}^{\mat{M}}(f-\operatorname{S}_{\mat{M}}f) |^q\text{.}
			\end{split}
		\end{equation*}
		In the remaining sum we first apply the aliasing formula \eqref{eq:aliasing}. Then, the H{\"o}lder inequality yields
		\begin{equation*}
			\begin{split}
				\sum\limits_{\vec{h}\in\mathcal G_{\text{S}}(\mat{M}^\T)}
				| \sigma_{\alpha}^{\mat{M}}(\vec{h})c_{\vec{h}}^{\mat{M}}(f-\operatorname{S}_{\mat{M}}f) |^q
				&=
				\sum_{\vec{h}\in\mathcal G_{\text{S}}(\mat{M}^\T)} \sigma_{\alpha q}^{\mat{M}}(\vec{h})
				\Biggl(\,
				\sum_{\vec{z}\in\mathbb Z^d\backslash\{\vec{0}\}}
				|c_{\vec{h}+\mat{M}^\T\vec{z}}(f)|
				\Biggr)^q\\
				&\leq
				\sum_{\vec{h}\in\mathcal G_{\text{S}}(\mat{M}^\T)}
					\sigma_{\alpha q}^{\mat{M}}(\vec{h})
					\Biggl(\,
						\sum_{\vec{z}\in\mathbb Z^d\backslash\{\vec{0}\}}
							\sigma_{-\mu p}^{\mat{M}}(\vec{h}+\mat{M}^\T\vec{z})
					\Biggr)^{q/p}\\&\quad
					\times
					\Biggl(\,
					\sum_{\vec{z}\in\mathbb Z^d\backslash\{\vec{0}\}}
						|\sigma_{\mu}^{\mat{M}}(\vec{h}+\mat{M}^\T\vec{z})
						c_{\vec{h}+\mat{M}^\T\vec{z}}(f)|^q
					\Biggr)\text{.}
				\end{split}	
			\end{equation*}
			The first sum over \( \vec{z} \) converges due to \( p\mu > d \), i.e., analogously to the proof of Theorem~\ref{theorem:approxtrig}, we get for \( \vec{h}\in\mathcal G_{\text{S}}(\mat{M}^\T) \)
			\enlargethispage{\baselineskip}
			\begin{equation*}
				\begin{split}
					\sum_{\vec{z}\in\mathbb Z^d\backslash\{\vec{0}\}} \sigma_{-\mu p}^{\mat{M}}(\vec{h}+\mat{M}^\T\vec{z})
					&=
					\sum_{\vec{z}\in\mathbb Z^d\backslash\{\vec{0}\}}
					\bigl(1+\|\mat{M}\|_2^2\|\mat{M}^{-\T}\vec{h}+\vec{z}\|_2^2\bigr)^{-p\mu/2}\\
					&\leq
					\sum_{\vec{z}\in\mathbb Z^d\backslash\{\vec{0}\}} \bigl(1+\|\mat{M}\|^2_2\bigl\||\vec{z}|-\tfrac{1}{2}\vec{1}\bigr\|_2^2\bigr)^{-p\mu/2}\\
					&\leq
					\sum_{\vec{z}\in\mathbb Z^d\backslash\{\vec{0}\}} \Biggl(\frac{\|\mat{M}\|_2^2}{4}\|2|\vec{z}|-\vec{1}\|_2^2\Biggr)^{-p\mu/2}\\
					&=
					\|\mat{M}\|_2^{-p\mu}
					2^{-p\mu}
					\sum_{\vec{z}\in\mathbb Z^d\backslash\{\vec{0}\}}\|2|\vec{z}|-\vec{1}\|^{-p\mu}_2\text{.}
				\end{split}
			\end{equation*}
		Using for \( \alpha \geq 0 \)
		\begin{align}\label{eq:maxGroupweight}
			\max_{\vec{h}\in\mathcal G_{\text{S}}(\mat{M}^\T)} \sigma_\alpha^{\mat{M}}(\vec{h})
			&\leq
			\Bigl(1+\|\mat{M}\|_2^2
			\bigl\|\mat{M}^{-\T}\mat{M}^\T\tfrac{1}{2}\vec{1}\bigr\|_2^2
			\Bigr)^{\alpha/2}\nonumber\\
			&=
			\bigl(1+\tfrac{d}{4}\|\mat{M}\|_2^2\bigr)^{\alpha/2}\nonumber\\
			&\leq(1+d)^{\alpha/2}\|\mat{M}\|_2^{\alpha}
			\text{,}
		\end{align}
		the upper bound for the last factor can be given as
		\begin{align*}
			\sum\limits_{\vec{h}\in\mathcal G_{\text{S}}(\mat{M}^\T)}&
				| \sigma_{\alpha}^{\mat{M}}(\vec{h})c_{\vec{h}}^{\mat{M}}(f-\operatorname{S}_{\mat{M}}f) |^q\\
				&\leq
			\sum_{\vec{h}\in\mathcal G_{\text{S}}(\mat{M}^\T)}
				\sigma_{\alpha q}^{\mat{M}}(\vec{h})
				2^{-\mu q}
				\|\mat{M}\|_2^{-\mu q}
			\Biggl(\,
				\sum_{\vec{z}\in\mathbb Z^d\backslash\{\vec{0}\}}\||2\vec{z}|-\vec{1}\|^{-p\mu}
			\Biggr)^{q/p}\\
			&\qquad
			\times
			\Biggl(\,
			\sum_{\vec{z}\in\mathbb Z^d\backslash\{\vec{0}\}}
				|\sigma_{\mu}^{\mat{M}}(\vec{h}+\mat{M}^\T\vec{z})c_{\vec{h}+\mat{M}^\T\vec{z}}(f)|^q
			\Biggr)\\
			&\leq
			2^{-\mu q}
			\|\mat{M}\|_2^{-\mu q}
			\biggl(\max_{\vec{h}\in\mathcal G_{\text{S}}(\mat{M}^\T)} \sigma_{q\alpha}^{\mat{M}}(\vec{h})\biggr)
			\Biggl(\,
				\sum_{\vec{z}\in\mathbb Z^d\backslash\{\vec{0}\}}\||2\vec{z}|-\vec{1}\|^{-p\mu}
			\Biggr)^{q/p}
			\\
			&\qquad\times
			\Biggr(\,
				\sum_{\vec{h}\in\mathcal G_{\text{S}}(\mat{M}^\T)}
				\sum_{\vec{z}\in\mathbb Z^d\backslash\{\vec{0}\}}
				|\sigma_{\mu}^{\mat{M}}(\vec{h}+\mat{M}^\T\vec{z})c_{\vec{h}+\mat{M}^\T\vec{z}}(f)|^q
			\Biggl)\\
			&\leq
			\|\mat{M}\|_2^{(\alpha-\mu) q}\gamma_{\mathrm{Sm}}^q\|f|A^{\mu}_{\mat{M},\,q}\|^q\text{.}\tag*{\qed}
		\end{align*}
		\let\qed\relax
	\end{proof}
	\begin{remark}
		It is easy to see that for a fundamental interpolant \( \operatorname{I}_{\mat{M}} \) satisfying the ellipsoidal periodic Strang-Fix conditions of order \( s \) for \( q \) and \( \alpha \) we have
		\[
		\gamma_{\mathrm{IP}} \leq C\cdot\gamma_{\mathrm{SF}}
		\]
		where the constant \( C \) depends on \( \mat{M},\alpha,s \) and \( q \) but is especially independent of \( f \).
	\end{remark}
	
	We summarize our treatment of the interpolation error in the following theorem.
	\begin{theorem}
		\label{theorem:gesamtapprox}
		Let an expanding regular matrix \( \mat{M}\in\mathbb Z^{d\times d} \) and a fundamental interpolant \( \operatorname{I}_{\mat{M}} \) fulfilling the periodic ellipsoidal Strang-Fix conditions of order \( s \) for \( q\geq1 \), and \( \alpha\geq 0 \) be given. Then for \( f\in A_{\mat{M},\,q}^{\mu}(\mathbb T^d)\), \( \mu\geq\alpha\geq 0 \) and \( \mu > d(1-1/q) \), we have
		\begin{equation*}
			\|f-\operatorname{L}_{\mat{M}}f|A_{\mat{M},\,q}^{\alpha}\|
			\leq C_{\rho}\left(\frac{1}{\|\mat{M}\|_2}\right)^{\rho}
			\|f|A_{\mat{M},\,q}^{\mu}\|\text{,}
		\end{equation*}
		where  \( \rho := \min\{s,\mu-\alpha\} \) and
		\begin{equation*}
			C_{\rho} :=
			\begin{cases}
				\gamma_{\mathrm{SF}} + 2^{\mu-\alpha} + \gamma_{\mathrm{IP}}\gamma_{\mathrm{Sm}}
				&\mbox{ if }\rho = s\text{,}\\
				(1+d)^{s+\alpha-\mu}
				\gamma_{\mathrm{SF}} + 2^{\mu-\alpha} + \gamma_{\mathrm{IP}}\gamma_{\mathrm{Sm}}
				&\mbox{ if }\rho = \mu-\alpha\text{.}
			\end{cases}
		\end{equation*}
	\end{theorem}
	\begin{proof}
		For $\rho = s$ Theorems \ref{theorem:trigapprox}--\ref{theorem:approxinterpol} can be applied directly due to \(\|f | A_{\mat{M},q}^{\alpha+s}\|\leq\|f | A_{\mat{M},q}^{\mu}\|\). If \(\rho = \mu-\alpha\), we have to replace Theorem~\ref{theorem:trigapprox} by an upper bound with respect to \( \mu \). Using this theorem and the inequality in \eqref{eq:maxGroupweight}, we get
		\begin{align*}
			\|&\operatorname{S}_{\mat{M}}f
			-\operatorname{L}_{\mat{M}}\operatorname{S}_{\mat{M}}f|A^{\alpha}_{\mat{M},\,q}\|\\
			&\leq  \gamma_{\mathrm{SF}}
			\|\mat{M}\|_2^{-s}
			\bigl\|
				\{\sigma_{\alpha+s}^{\mat{M}}(\vec{h})c_{\vec{h}}(f)\}_{\vec{}\in\mathcal G_{\text{S}}(\mat{M}^\T)}
				\bigr|\ell_q(\mathcal G_{\text{S}}(\mat{M}^\T))
			\bigr\|
			\\
			&\leq
			\gamma_{\mathrm{SF}}
			\max_{\vec{h}\in\mathcal G_{\text{S}}(\mat{M}^\T)} \sigma_{\alpha+s-\mu}^{\mat{M}}(\vec{h})
			\|\mat{M}\|_2^{-s}
			\bigl\|
				\{
					\sigma_{\mu}^{\mat{M}}(\vec{h})
					c_{\vec{h}}(f)
				\}_{\vec{}\in\mathcal G_{\text{S}}(\mat{M}^\T)}
			\bigr|
			\ell_q(\mathcal G_{\text{S}}(\mat{M}^\T))
			\bigr\|
			\\
			&\leq
			\gamma_{\mathrm{SF}}
			(1+d)^{s+\alpha-\mu}
			\|\mat{M}\|_2^{s+\alpha-\mu}
			\
			\|\mat{M}\|_2^{-s}
			\bigl\|
			f
			\bigr|
			A_{\mat{M},q}^\mu
			\bigr\|
			\text{.}\tag*{\qed}
		\end{align*}
		\let\qed\relax
	\end{proof}
	\begin{remark}\label{rem:convergence}
		The factor \( \kappa_{\mat{M}}^{-s} \) in both constraints of the Strang-Fix conditions, cf. De\-fi\-ni\-ti\-on~\ref{def:SFC}, enforces a strong decay on the Fourier coefficients of the fundamental interpolant \( \operatorname{I}_\mat{M} \). Omitting this factor in both constraints, i.e., leaving just \(\|\mathbf{M}\|_2^{-\alpha}\) in the second one, weakens to a less restrictive constraint on the fundamental interpolant \( \operatorname{I}_{\mat{M}} \). This
		changes the decay rate from
		\[
		\Biggl(\frac{1}{\|\mat{M}\|_2}\Biggr)^s
		= \Biggl(\frac{1}{\sqrt{\lambda_d(\mat{M}^\T\mat{M})}}\Biggr)^s
		\]
		to
		\[
		\Biggl(\frac{
		\kappa_{\mat{M}}}{\|\mat{M}\|_2}\Biggr)^s
		=
		\bigl(\|\mat{M}^{-\T}\|_2\bigr)^s
		=
		\Biggl(\frac{
		1}{\sqrt{\lambda_1(\mat{M}^\T\mat{M})}}\Biggr)^s\text{,}
		\]
		which is then also the rate of decay in Theorem~\ref{theorem:gesamtapprox}. Hence, while this formulation eases the constraints with respect to the decay by restricting it just to the shortest axis of the ellipsoid given by \( \|\mat{M}^{-\T}\circ\|_2=1 \), the rate of convergence is also relaxed. On the other hand the stronger formulation in Theorems~\ref{theorem:trigapprox}--\ref{theorem:gesamtapprox} requires the fundamental interpolant to fulfill stronger constraints.

		When increasing the number of sampling points, i.e., the determinant \( |\det{\mat{M}}| \), for both variations there are cases where the bound is not decreased. Namely, in the first one if the value \( \|\mat{M}\|_2 \) is not increased, in the second version if the value \( \|\mat{M}^{-\T}\|_2 \) is not decreased.

		Again, for the tensor product case \( \mat{M}=\operatorname{diag}(N,\ldots,N) \) and the one-dimensional setting, both formulations of the Strang-Fix conditions and the resulting error bounds are equal.
	\end{remark}

\section{The 3-directional Box Splines}\label{sec:Example} 
		For a function \( g\in \operatorname{L}_1(\mathbb R^d) \) on the Euclidean space \( \mathbb R^d \) the Fourier transform is given by
		\[
			\hat g(\vec{\xi}) := \int_{\mathbb R^d} g(\vec{x})\E^{\I\vec{\xi}^\T\vec{x}}\,\D\vec{x},\quad \vec{\xi}\in\mathbb R^d,
		\]
		and we introduce the periodization with respect to a regular matrix \( \mat{M}\in\mathbb Z^{d\times d} \) for a function \( g:\mathbb R^d \to \mathbb C \) having compact support
		\[
			g^\mat{M}: \mathbb T^d\to\mathbb C,\quad g^\mat{M} := \sum_{\vec{z}\in\mathbb Z^d} g\bigl(\tfrac{1}{2\pi}\mat{M}(\circ - 2\pi\vec{z})\bigr)\text{.}
		\]
		Its Fourier coefficients \( c_{\vec{k}}(g^\mat{M})\), \(\vec{k}\in\mathbb Z^d \), can be obtained from \( \hat g \) by using the substitution \( \vec{y} = \tfrac{1}{2\pi}\mat{M}\vec{x} \), i.e., \( \,\D\vec{y} = \tfrac{1}{(2\pi)^{d}}m\,\D\vec{x}\). Hence
		\begin{equation*}
			\begin{split}
				c_{\vec{k}}(g^\mat{M}) &= \frac{1}{(2\pi)^d}\int_{\mathbb T^d } \sum_{\vec{z}\in\mathbb Z^d}g\bigl(\tfrac{1}{2\pi}\mat{M}(\vec{x} - 2\pi\vec{z})\bigr)\E^{\I\vec{k}^\T\vec{x}}\,\D\vec{x}
				\\
				&=\frac{1}{(2\pi)^d} \sum_{\vec{z}\in\mathbb Z^d} \int_{\mathbb T^d }g\bigl(\tfrac{1}{2\pi}\mat{M}(\vec{x} - 2\pi\vec{z})\bigr)\E^{\I\vec{k}^\T\vec{x}}\,\D\vec{x}\\
				&= \frac{1}{m}\int_{\mathbb R^d }g(\vec{y})\E^{\I\vec{k}^\T\bigl(2\pi\mat{M}^{-1}\vec{y}\bigr)}\,\D\vec{y}\\
				&= \frac{1}{m}\hat g \bigl(2\pi\mat{M}^{-\T}\vec{k}\bigr)\text{.}
			\end{split}
		\end{equation*}
		The same applied to the Lagrange interpolation symbol \( \tilde g(\vec{\xi}) := \sum_{\vec{z}\in\mathbb Z^d}g(\vec{z})\E^{\I\vec{\xi}^\T\vec{z}} \) yields
\( c_{\vec{h}}^\mat{M}(\tilde g^\mat{M}) = \frac{1}{m}\tilde g(2\pi\mat{M}^{-\T}\vec{h}) \), \( \vec{h}\in\mathcal G(\mat{M}^{\T}) \).

		We look at an example for the case \( d=2 \). The 3-directional box splines \( B_{\vec{p}} \), \( \vec{p} = (p_1,p_2,p_3)\in\mathbb N^3\), \( p_j \geq 1\), \( j=1,2,3 \), are given by their Fourier transform
		\[
		\hat B_{\vec{p}}(\vec{\xi}) := \bigl(\operatorname{sinc} \tfrac{1}{2}\xi_1\bigr)^{p_1}\bigl(\operatorname{sinc} \tfrac{1}{2}\xi_2\bigr)^{p_2}\bigl(\operatorname{sinc} \tfrac{1}{2}(\xi_1+\xi_2)\bigr)^{p_3}\text{.}
		\]
		Applying the periodization, we obtain the function \( B_{\vec{p}}^\mat{M}: \mathbb T^2\to \mathbb C \) by its Fourier coefficients
		\[
		c_{\vec{k}}(B_{\vec{p}}^\mat{M}) = \tfrac{1}{m}\bigl(\operatorname{sinc} \pi\vec{k}^\T\mat{M}^{-1}\vec{e}_1)^{p_1} \bigl(\operatorname{sinc} \pi\vec{k}^\T\mat{M}^{-1}\vec{e}_2)^{p_2} \bigl(\operatorname{sinc} \pi\vec{k}^\T\mat{M}^{-1}(\vec{e}_1+\vec{e}_2)\bigr)^{p_3}\text{.}
		\]
		Due to positivity of \( \hat B_{\vec{p}}(\vec{\xi}) \), \( \vec{\xi}\in[-\pi,\pi]^2 \), cf. \cite[Section 4]{deBoor:1985}, we know that \( c_{\vec{h}}(B_{\vec{p}}^{\mat{M}})\neq 0 \) for \( \vec{h}\in\mathcal G(\mat{M}^\T) \). Hence by \cite[Corollary 3.5]{LangemannPrestin:2010} the translates \( \T_{\vec{y}}B_{\vec{p}}^\mat{M} \), \( \vec{y}\in\mathcal P(\mat{M}) \), form a basis of \( V_\mat{M}^{B_{\vec{p}}^\mat{M}} \).
		\begin{theorem}\label{thm:3dirBoxspline}
			Let \( \mat{M}\in\mathbb Z^{2\times 2} \) be a regular matrix, \( \vec{p}\in\mathbb N^3 \), \( p_j\geq1 \), \( j=1,2,3 \) a vector, \( s := \min\{p_1+p_2,p_1+p_3,p_2+p_3\} \) and \( \alpha\geq 0 \), \( q\geq 1 \), such that \( s-\alpha > 2 \).
			
			The fundamental interpolant \( \operatorname{I}_{\mat{M}}\in V_{\mat{M}}^{B_{\vec{p}}^{\mat{M}}} \) of the periodized 3-directional box spline \( B_{\vec{p}}^{\mat{M}} \) fulfills the periodic ellipsoidal Strang-Fix conditions of order \( s-\alpha \) for \( \alpha \) and \( q \), which depends on \( \vec{p} \).
		\end{theorem}
		\begin{proof}
			We first examine the case \( \alpha = 0\). Taking a look at the second Strang-Fix condition, we obtain for \( \vec{h}\in\mathcal G_{\text{S}}(\mat{M}^{\T}) \) and \( \vec{z}\in\mathbb Z^d\backslash\{\vec{0}\} \), following the same steps as in the proof of Theorem 1.10 in \cite{Poe95}, the inequality
			\begin{equation*}
				\begin{split}
					|mc_{\vec{h}+\mat{M}^{\T}\vec{z}}(\operatorname{I}_\mat{M})|
					&= \Biggl|
					 \frac{c_{\vec{h}+\mat{M}^{\T}\vec{z}}(B_{\vec{p}}^\mat{M})}{c_{\vec{h}}^\mat{M}(B_{\vec{p}}^\mat{M})}
					\Biggr|\\
					&\leq
					\frac{1}{c_{\vec{h}}^\mat{M}(B_{\vec{p}}^\mat{M})}
					\frac{|\sin \pi\vec{h}^{\T}\mat{M}^{-1}(\vec{e}_1+\vec{e}_2)|^{p_3}}{|\pi(\mat{M}^{-\T}\vec{h} + \vec{z})^\T(\vec{e}_1+\vec{e}_2)|^{p_3}}
					\prod_{j=1}^2\frac{|\sin \pi\vec{h}^{\T}\mat{M}^{-1}\vec{e}_j|^{p_j}}{|\pi(\mat{M}^{-\T}\vec{h} + \vec{z})^\T\vec{e}_j|^{p_j}}
					\\
					&\leq
					\frac{1}{c_{\vec{h}}^\mat{M}(B_{\vec{p}}^\mat{M})}
					\frac{|\vec{h}^{\T}\mat{M}^{-1}(\vec{e}_1+\vec{e}_2)|^{p_3}}{|(\mat{M}^{-\T}\vec{h} + \vec{z})^\T(\vec{e}_1+\vec{e}_2)|^{p_3}}
					\prod_{j=1}^2\frac{|\vec{h}^{\T}\mat{M}^{-1}\vec{e}_j|^{p_j}}{|(\mat{M}^{-\T}\vec{h} + \vec{z})^\T\vec{e}_j|^{p_j}}\text{.}
				\end{split}
			\end{equation*}
			For
			\( z_1\neq 0 \), \( z_2\neq 0 \), \( z_1+z_2\neq 0 \), and \( |z_1+z_2|\neq 1 \) it holds using \(| (\mat{M}^{-\T}\vec{h}+\vec{z})^\T\vec{e}_j| \geq (|z_j|-\tfrac{1}{2}) \), \( j=1,2 \), and \( | (\mat{M}^{-\T}\vec{h} + \vec{z})^\T(\vec{e}_1+\vec{e}_2 |) \geq (|z_1+z_2|-1) \) that
			\begin{equation*}
				|mc_{\vec{h}+\mat{M}^{\T}\vec{z}}(B_{\vec{p}}^\mat{M})|
				\leq
				\Bigl(
				|\vec{h}^\T\mat{M}^{-1}\vec{e}_1| + |\vec{h}^\T\mat{M}^{-1}\vec{e}_2|
				\Bigr)^{\!s}
				\frac{1}{(|z_1+z_2|-1)^{p_3}}
				\prod_{j=1}^2
				\frac{1}{(|z_j|-\tfrac{1}{2})^{p_j}}
				\text{,}
			\end{equation*}
			where applying the Cauchy-Schwarz inequality \( |h_1| + |h_2| \leq \sqrt{2}\|\vec{h}\|_2 \) yields
			\begin{equation*}
				|mc_{\vec{h}+\mat{M}^{\T}\vec{z}}(B_{\vec{p}}^\mat{M})|
				\leq
				\|\mat{M}^{-\T}\vec{h}\|_2^s
				\frac{2^{s/2}}
				{(|z_1|-\tfrac{1}{2})^{p_1}
				(|z_2|-\tfrac{1}{2})^{p_2}
				(|z_1+z_2|-1)^{p_3}}\text{.}
			\end{equation*}
%
%
			Defining \[
			 A := \biggl( \min_{\vec{h}\in\mathcal G_{\text{S}}(\mathbf{M}^\T)} mc_{\vec{h}}^{\mathbf{M}}(B_{\vec{p}}^{\mathbf{M}})\biggr)^{-1}
			\]
			we can use the last inequality to obtain that the fundamental interpolant \(\operatorname{I}_{\mathbf{M}}\) corresponding to \(B_{\vec{p}}^{\mathbf{M}}\) fulfills the Strang-Fix conditions of order \(s\) with \(\alpha=0\), where the series for \(\gamma_{\mathrm{SF}}\) is given by
			\[
				b_{\vec{z}}
				=
				b^0_{\vec{z}}
				=\frac{2^{s/2}A}
				{(|z_1|-\tfrac{1}{2})^{p_1}
				(|z_2|-\tfrac{1}{2})^{p_2}
				(|z_1+z_2|-1)^{p_3}}\text{,}
			\]
			at least for \(\vec{z} = (z_1,z_2)^\T\) with \(z_1\neq 0\), \(z_2\neq 0\), \(z_1+z_2\neq 0\) and \(|z_1+z_2|\neq 1\). An upper bound for the remaining indices \(\vec{z}\) can be established using similar arguments as for this case. These estimates can be directly transcribed from the already mentioned proof, cf.~\cite[pp. 51-57]{Poe95}, including the bound for the first of the Strang-Fix conditions, i.e., \(b^0_{\vec{0}}\). This concludes the proof for the case \( \alpha=0 \). 

For \(\alpha\geq 0\) we define the series \( b_{\vec{z}} := 2^{-\alpha/2}\|\mathbf{M}^{-\T}\|_2^{-\alpha}b^0_{\vec{z}}\), \( \vec{z}\in\mathbb Z^2 \), and obtain for the first Strang-Fix condition with \( \vec{h}\in\mathcal G_{\text{S}}(\mat{M}^\T) \) and using \(\|\mathbf{M}\|_2^{\alpha} \geq 1\) that
			\begin{align*}
				|1-mc_{\vec{h}}(\operatorname{I}_{\mathbf{M}})|
				&\leq b^0_{\vec{0}}\kappa_{\mathbf{M}}^{-s}\|\mat{M}^{-\T}\vec{h}\|_2^s\\
				&\leq \kappa_{\mathbf{M}}^{-\alpha}2^{-\alpha/2}b^0_{\vec{0}}\kappa_{\mathbf{M}}^{-(s-\alpha)}\|\mathbf{M}^{-\T}\vec{h}\|_2^{s-\alpha}\\
				& \leq b_{\vec{0}}\kappa_{\mathbf{M}}^{-(s-\alpha)}\|\mathbf{M}^{-\T}\vec{h}\|_2^{s-\alpha}\text{.}
			\intertext{For the second condition we get}
				|mc_{\vec{h}+\mat{M}^{\T}\vec{z}}(\operatorname{I}_{\mathbf{M}})|
				&\leq b^0_{\vec{z}}\kappa_{\mathbf{M}}^{-\alpha}\|\mathbf{M}^{-\T}\vec{h}\|_2^{\alpha} 					 \kappa_{\mathbf{M}}^{-(s-\alpha)}\|\mat{M}^{-\T}\vec{h}\|_2^{s-\alpha}\\
				&\leq \|\mathbf{M}^{-\T}\|_2^{-\alpha} 2^{-\alpha/2}b^0_{\vec{z}}\|\mathbf{M}\|_2^{-\alpha}\kappa_{\mathbf{M}}^{-(s-\alpha)}\|\mat{M}^{-\T}\vec{h}\|_2^{s-\alpha}\\
				&\leq b_{\vec{z}}\|\mathbf{M}\|_2^{-\alpha}\kappa_{\mat{M}}^{-(s-\alpha)}\|\mat{M}^{-\T}\vec{h}\|_2^{s-\alpha}
				\text{,}
			\end{align*}
			where the first inequality in both cases is mentioned for completeness. The series which is used to define \( \gamma_{\mathrm{SF}} \) is given by
			\[
			 \gamma_{\mathrm{SF}}^q = \sum_{\vec{z}\in\mathbb Z^2} |(1+\|\mat{M}\|_2^2\|\mat{M}^{-\T}\vec{z}\|_2^2)^{\alpha/2})b_{\vec{z}}|^q
			\text{,}
			\]
			which converges for \( s-\alpha >2 \) by applying again the same inequalities that were used for the case of a diagonal matrix \( \mat{M} = \operatorname{diag}(N,\ldots,N) \), \( q=2 \), and \( \alpha = 0 \) in Theorem 1.10 of \cite{Poe95}.
		\end{proof}
		This can also be applied to the \( d \)-variate case, \( d>2 \), using the \( \tfrac{d(d+1)}{2} \)-directional box spline \( B_{\vec{p}}\), \(\vec{p}\in \mathbb N^{\frac{d(d+1)}{2}} \), consisting of the directions \( \vec{e}_j \), \( j=1,\ldots,d \), and \( \vec{e}_j+\vec{e}_i \), \( i,j=1,\ldots,d, i\neq j \), the corresponding 4-directional box spline \cite[Thm. 1.11]{Poe95}, and its multivariate version, the \( d^2 \)-directional box spline, which can be generated analogously to the \( \tfrac{d(d+1)}{2} \)-directional box spline, i.e., using the directions \( \vec{e}_j \), \( \vec{e}_i+\vec{e}_j \) and \( \vec{e}_i-\vec{e}_j \). Nevertheless, for the periodized \( d^2 \)-directional box spline \( B_{\vec{q}}^{\mat{M}} \), \( \vec{q}\in\mathbb N^{d^2} \), the fundamental interpolant \( \operatorname{I}_{\mat{M}} \) does not exist. This can be seen by looking at \(B_{\vec{q}}^{\mathbf{M}}\) in the Fourier domain, where it does contain at least one two-dimensional 4-directional box spline as a factor. Hence the non-normal interpolation of the 4-directional box spline, which was investigated in~\cite{Jetter:1991} carries over to the higher dimensional case. In order to apply the above mentioned theorems, we have to use the so called incorrect interpolation, i.e., we set \( c_{\vec{h}}(\operatorname{I}_{\mat{M}}) = m^{-1} \) for \( \vec{h}\in\mathcal G_{\text{S}}(\mat{M}^\T) \), where \( c_{\vec{h}}^{\mat{M}}(B_{\vec{q}}^{\mat{M}}) = 0 \).

\end{document}